\documentclass[11pt]{amsart}
\usepackage{upkg}
\title{Newton polytopes of fireworks Grothendieck polynomials}

\author{Jack Chen-An Chou}
\address{Jack Chen-An Chou, School of Mathematics, University of Minnesota, Minneapolis, MN 55455.}
\email{chou0188@umn.edu}

\author{Linus Setiabrata}
\address{Linus Setiabrata, Department of Mathematics, Massachusetts Institute of Technology, Cambridge, MA 02139.} \email{setia@mit.edu}
\begin{document}

\maketitle

\begin{abstract}
We show that the support of the Grothendieck polynomial $\mathfrak G_w$ of any fireworks permutation is as large as possible: a monomial appears in $\mathfrak G_w$ if and only if it divides $\mathbf x^{\wt(\overline{D(w)})}$ and is divisible by some monomial appearing in the Schubert polynomial $\mathfrak S_w$. Our formula implies that the homogenization of $\mathfrak G_w$ has M-convex support. We also show that for any fireworks permutation $w\in S_n$, there exists a layered permutation $\pi(w)\in S_n$ so that $\supp(\mathfrak G_{\pi(w)})\supseteq \supp(\mathfrak G_w)$.
\end{abstract}

\section{Introduction}

Schubert polynomials $\mathfrak S_w$ and Grothendieck polynomials $\mathfrak G_w$ are families of multivariate polynomials, introduced by Lascoux and Sch\"utzenberger \cite{ls82a,ls82b}, which represent Schubert varieties in the cohomology and $K$-theory of the flag variety respectively. These polynomials are of great interest due to their rich connections to geometry, algebra, and representation theory. At the same time, many central problems are amenable to combinatorial methods \cite{bb93,bjs93,fk94,fk96,km05}. Compared to Schubert polynomials, Grothendieck polynomials encode finer information, but their combinatorial structure is less understood.

It has been fruitful to study Schubert polynomials from a polytopal point of view. The Newton polytope of $\mathfrak S_w$ is the \emph{Schubitope} $\mathcal S_{D(w)}$ associated to the Rothe diagram $D(w)$ \cite{mty19}. The set of Schubitopes has good closure properties, and any Schubitope can be computed ``column by column'' via a Minkowski sum decomposition $\mathcal S_D = \mathcal S_{D_1} + \dots + \mathcal S_{D_m}$.  Using this decomposition, Fink--M\'esz\'aros--St.\ Dizier \cite{fms18} showed that Schubert polynomials have saturated Newton polytope (SNP) and Adve--Robichaux--Yong \cite{ary21} gave a simple combinatorial criterion to determine if a given monomial appears in $\mathfrak S_w$. Other applications of the Schubitope perspective include a sufficient condition for vanishing of Littlewood--Richardson coefficients \cite{sy22} and lower bounds on the principal specializations $\mathfrak S_w(1,...,1)$ \cite{mst21,gl24}.

The corresponding story for Grothendieck polynomials is expected to be similarly well-behaved. For example, the \emph{support} of $\mathfrak G_w$ is conjecturally \cite[Conj 1.1, Conj 1.3]{mss25} governed by its top- and bottom-degree homogeneous components via the formula
\begin{equation}
\label{eqn:master-support}
\supp(\mathfrak G_w) = \bigcup_{\substack{\alpha\in\supp(\mathfrak S_w)\\\beta\in\supp(\mathfrak G_w^{\mathrm{top}})}}[\alpha,\beta],
\end{equation}
where $[\alpha,\beta]$ denotes the componentwise-comparison interval $\{\gamma\in\ZZ^n\colon \alpha_i \leq \gamma_i \leq \beta_i\textup{ for all } i\}$.

Every monomial appearing in any Grothendieck polynomial $\mathfrak G_w$ divides $\mathbf x^{\wt(\overline{D(w)})}$, where $\overline{D(w)}$ is the \emph{upwards-closure} of $D(w)$. If $w$ is a \emph{fireworks} permutation, then $\mathfrak G_w^{\mathrm{top}}$ is a scalar multiple of $\mathbf x^{\wt(\overline{D(w)})}$. The main result of this paper is that~\eqref{eqn:master-support} holds for fireworks Grothendieck polynomials.
\begin{thm}
\label{thm:main}
Let $w$ be a fireworks permutation. If $\mathbf x^\alpha$ appears in $\mathfrak G_w$ and $x_i\cdot \mathbf x^\alpha$ divides $\mathbf x^{\wt(\overline{D(w)})}$, then $x_i\cdot \mathbf x^\alpha$ appears in $\mathfrak G_w$. In particular,
\[
\supp(\mathfrak G_w) = \bigcup_{\alpha\in\supp(\mathfrak S_w)}[\alpha,\wt(\overline{D(w)})].
\]
\end{thm}
Theorem~\ref{thm:main} implies that the Newton polytope of a fireworks Grothendieck polynomial $\mathfrak G_w$ can be computed ``column by column'', like with Schubitopes, via the formula
\[
\supp(\mathfrak G_w)= \left(\sum_{j=1}^nP_{\mathrm{sp}}(\SM(D_j))\right)\cap\ZZ^n,
\]
where $D_j$ are the columns of the Rothe diagram $D(w)$ and $P_{\mathrm{sp}}(\SM(D_j))$ is the Schubert spanning set polytope (see Definition~\ref{defn:schubert-spanning-set-polytope}).

Our motivation for studying fireworks permutations in particular came from the following guiding principle: according to~\eqref{eqn:master-support}, the bottom degree part of $\supp(\mathfrak G_w)$ is simple, the top degree part of $\supp(\mathfrak G_w)$ is small, and these two components determine the support. In the fireworks case, $\supp(\mathfrak G_w)$ is as simple as possible, and in fact, having the goal formula~\eqref{eqn:master-support} guided and motivated our proof of Theorem~\ref{thm:main}.

The description afforded by Theorem~\ref{thm:main} implies that homogenized Grothendieck polynomials of fireworks permutations have M-convex support.
\begin{cor}
\label{cor:m-convex}
When $w$ is fireworks, the homogenized Grothendieck polynomial $\widetilde{\mathfrak G}_w$ has M-convex support, i.e., $\widetilde{\mathfrak G}_w$ has SNP and its Newton polytope is a generalized permutahedron. In particular, $\mathfrak G_w$ has SNP, each degree component $\mathfrak G_w^{(c)}$ also has M-convex support, and the sizes $|\supp(\mathfrak G_w^{(c)})|$ of their supports form a log-concave sequence.
\end{cor}
Homogenized Grothendieck polynomials $\widetilde{\mathfrak G}_w$ are known to also have M-convex support when $w$ is vexillary (\cite{hafner22,hmss24}) or when $\mathfrak S_w$ has coefficients all zero or one (\cite{ccmm22}). The general case remains open (\cite{hmms22}). As far as the authors are aware, M-convexity of $\supp(\widetilde{\mathfrak G}_w)$ was previously known for only exponentially (in $n$) many permutations (cf. \cite{fms21}). Corollary~\ref{cor:m-convex} supplies superexponentially many permutations for which $\supp(\widetilde{\mathfrak G}_w)$ is M-convex.

Grothendieck polynomials are generating functions for certain combinatorial objects called pipe dreams. The key ingredient in our proof of Theorem~\ref{thm:main} is an explicit pipe dream construction. Specifically, Theorem~\ref{thm:main} follows from repeated application of the following claim.

\begin{thm}
\label{thm:main2-steps}
Let $w\in S_n$ be a fireworks permutation and let $P \in \PD(w)$ be a pipe dream with $\wt(P) < \wt(\overline{D(w)})$. For any $a$ with $\wt(P)_a < \wt(\overline{D(w)})_a$, there exists a pipe dream $Q \in \PD(w)$ satisfying:
\begin{itemize}
\item $\wt(Q)_b = \wt(P)_b$ for $b<a$
\item $\wt(Q)_a = \wt(P)_a + 1$
\item $\wt(Q)_b \leq \wt(P)_b$ for $b>a$.
\end{itemize}
\end{thm}

A motivation for our proof of Theorem~\ref{thm:main2-steps} is the observation (Corollary~\ref{cor:good-pipes-only}) that, for fireworks permutations, the weights of maximal-degree pipe dreams are in some sense ``inclusion-maximal'' (rather than just componentwise- or degree-maximal).

A special subclass of fireworks permutations called \emph{layered} permutations has previously appeared in various maximization problems in Schubert calculus. Layered permutations are conjectured \cite{app26} to asymptotically maximize the value $\mathfrak S_w(1, \dots, 1)$. The same paper \cite{app26} (see also \cite{mpp19}) disproved an earlier conjecture \cite{ms16} that layered permutations genuinely maximize $\mathfrak S_w(1, \dots, 1)$. The maximal value of the principal specialization $\mathfrak G^{(\beta=1)}_w(1,\dots,1)$ of the $\beta=1$ Grothendieck polynomial is asymptotically achieved at a layered permutation \cite{mppy24}. The maximal value of $\deg(\mathfrak G_w) - \deg(\mathfrak S_w)$ is achieved at a layered permutation \cite{psw24}. The size $|\supp(\mathfrak S_w)|$ of the support of $\mathfrak S_w$ is conjecturally maximized at layered permutations \cite{gl24}.

We extend this list with a corollary of similar flavor. The support formula in Theorem~\ref{thm:main} implies that layered permutations maximize Newton polytopes of fireworks Grothendieck polynomials with respect to inclusion.
\begin{cor}
\label{cor:maximal-newton}
Let $w$ be a fireworks permutation and write $\pi(w)$ for the layered permutation whose block sizes are the lengths of descending runs of $w$. Then $\supp(\mathfrak G_{\pi(w)})\supseteq\supp(\mathfrak G_w)$.
\end{cor}

\section*{Acknowledgements}
We thank David Anderson, Ian Cavey, Zachary Hamaker, Tomoo Matsumura, Karola M\'esz\'aros, Alejandro Morales, Greta Panova, Brendon Rhoades, Avery St.~Dizier, Alexander Yong, and Tianyi Yu for many helpful and inspiring conversations. We especially thank Avery St.~Dizier, Zachary Hamaker, Tianyi Yu, and the anonymous referees for a careful reading of an earlier draft, and Alejandro Morales, Avery St.~Dizier, and Alexander Yong for pointing out useful references which improved the exposition of this paper. J.\ Chou was partially supported by NSF Grant DMS-2054423.

\section{Background}
\subsection*{Grothendieck polynomials and pipe dreams} For any $i\in[n-1]$, the {\color{darkred}\emph{divided difference operator}} $\del_i\colon \ZZ[x_1, \dots, x_n]\to\ZZ[x_1, \dots, x_n]$ is defined to be
\[
\del_i(f)\colonequals \frac{f - s_if}{x_i - x_{i+1}},
\]
where $s_i$ is the operator switching the variables $x_i$ and $x_{i+1}$, and the {\color{darkred}\emph{isobaric divided difference operator}} $\overline\del_i\colon \ZZ[x_1, \dots, x_n] \to \ZZ[x_1, \dots, x_n]$ is defined to be $\overline\del_i(f)\colonequals \del_i((1-x_{i+1})f)$.

The {\color{darkred}\emph{Schubert polynomial}} $\mathfrak S_w(\mathbf x)$ and {\color{darkred}\emph{Grothendieck polynomial}} $\mathfrak G_w(\mathbf x)$ of $w\in S_n$ are defined by the recursions
\[
\mathfrak S_w(\mathbf x) = \begin{cases} x_1^{n-1}\dots x_{n-1}&\textup{ if } w = w_0,\\\del_i(\mathfrak S_{ws_i}(\mathbf x))&\textup{ if } \ell(w) < \ell(ws_i),\end{cases}\quad\textup{ and } \quad \mathfrak G_w(\mathbf x) = \begin{cases} x_1^{n-1}\dots x_{n-1} &\textup{ if } w = w_0,\\ \overline\del_i(\mathfrak G_{ws_i}(\mathbf x)) &\textup{ if } \ell(w) < \ell(ws_i).\end{cases}
\]
The Schubert polynomial is the lowest degree part of the corresponding Grothendieck polynomial. Schubert and Grothendieck polynomials can be computed using combinatorial objects called pipe dreams.

A {\color{darkred}\emph{pipe dream}} is a tiling of a staircase grid $\{(i,j)\in[n]\times[n]\colon i + j \leq n\}$ using {\color{darkred}\emph{cross tiles}} $\ptile$ and {\color{darkred}\emph{bump tiles}} $\bumptile$. By placing half-bump tiles $\jtile$ on the antidiagonal $\{(i,j)\in[n]\times[n]\colon i+j = n+1\}$, a pipe dream can be viewed as a network of $n$ pipes running from the top of the grid to the left.

We label the pipes of a pipe dream $P$ with the numbers $1$ through $n$ along the top edge. Tracing the pipes as they travel from the top of the grid to the left and ignoring any crossings between any pair of pipes which have already crossed (i.e., replacing redundant cross tiles with bump tiles) yields a string of numbers $(\del(P))(1), (\del(P))(2), \dots, (\del(P))(n)$ that defines a permutation $\del(P) \in S_n$.

\begin{defn}
We say that a cross tile $\ptile$ of a pipe dream is an {\color{darkred}\emph{active crossing}} if its two pipes really cross, i.e., if the pipe entering in the top exits from the bottom. Otherwise, we say the cross tile $\ptile$ is an {\color{darkred}\emph{inactive crossing}}.
\end{defn}

\begin{figure}[ht]
\includegraphics[scale=1.3]{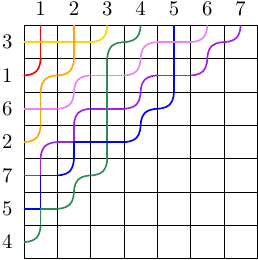}
\caption{A pipe dream $P$ for the fireworks permutation $w = 3162754$, with pipes colored by their label. The pipe dream $P$ is nonreduced, its weight is $(3,2,2,2,1,1,0)$, and it has inactive crossings in $(4,2)$, $(5,1)$, and $(6,1)$.}
\end{figure}

For $w\in S_n$, we write $\PD(w)$ for the set of pipe dreams whose associated permutation is $w$.

A pipe dream $P$ is {\color{darkred}\emph{reduced}} if the number of cross tiles is equal to $\ell(\del(P))$ (i.e., no inactive crossings). 

The {\color{darkred}\emph{weight}} of a pipe dream $P$ is the integer vector $\wt(P)\in\ZZ_{\geq 0}^n$ whose $i$-th component records the number of crossing tiles in row $i$.

Pipe dreams give a combinatorial model for Schubert and Grothendieck polynomials:

\begin{thm}[{\cite[Thm 2.3]{fk94}; see also \cite[Cor 5.4]{km04}, \cite[Thm 6.1]{weigandt21}}]
\label{thm:pd-generates}
The Schubert and Grothendieck polynomials are given by the formula
\begin{align*}
\mathfrak S_w &= \sum_{\substack{P\in\PD(w),\\ P\textup{ reduced}}}\mathbf x^{\wt(P)},\\
\mathfrak G_w &= \sum_{P\in\PD(w)}(-1)^{|P| - \ell(w)}\mathbf x^{\wt(P)}.
\end{align*}
\end{thm}

The first equality above follows from the second. See \cite{fk96,bb93,km05} for earlier work addressing pipe dreams and Schubert polynomials.

\begin{exa}
The set $\PD(2413)$ contains three pipe dreams:
$$
\begin{tikzpicture}[x=1em,y=1em,thick,rounded corners, color = blue]
\draw[step=1,gray,ultra thin] (0,0) grid (4,4);
\draw[color=blue] (0,0.5)--(0.5, 0.5)--(0.5, 1.5)--(1.5, 1.5)--(1.5, 3.5)--(2.5, 3.5)--(2.5, 4);
\draw[color=blue] (0,1.5)--(0.5, 1.5)--(0.5, 4);
\draw[color=blue] (0,2.5)--(2.5, 2.5)--(2.5, 3.5)--(3.5, 3.5)--(3.5, 4);
\draw[color=blue] (0,3.5)--(1.5, 3.5)--(1.5, 4);
\end{tikzpicture}
\quad\quad\quad
\begin{tikzpicture}[x=1em,y=1em,thick,rounded corners, color = blue]
\draw[step=1,gray,ultra thin] (0,0) grid (4,4);
\draw[color=blue] (0,0.5)--(0.5, 0.5)--(0.5, 1.5)--(1.5, 1.5)--(1.5, 2.5)--(2.5, 2.5)--(2.5, 4);
\draw[color=blue] (0,1.5)--(0.5, 1.5)--(0.5, 4);
\draw[color=blue] (0,2.5)--(1.5, 2.5)--(1.5, 3.5)--(3.5, 3.5)--(3.5, 4);
\draw[color=blue] (0,3.5)--(1.5, 3.5)--(1.5, 4);
\end{tikzpicture}
\quad\quad\quad
\begin{tikzpicture}[x=1em,y=1em,thick,rounded corners, color = blue]
\draw[step=1,gray,ultra thin] (0,0) grid (4,4);
\draw[color=blue] (0,0.5)--(0.5, 0.5)--(0.5, 1.5)--(1.5, 1.5)--(1.5, 3.5)--(3.5, 3.5)--(3.5, 4);
\draw[color=blue] (0,1.5)--(0.5, 1.5)--(0.5, 4);
\draw[color=blue] (0,2.5)--(2.5, 2.5)--(2.5, 3.5)--(2.5, 4);
\draw[color=blue] (0,3.5)--(1.5, 3.5)--(1.5, 4);
\end{tikzpicture}
$$
The first two pipe dreams are reduced and the last pipe dream is not, so Theorem~\ref{thm:pd-generates} implies that
\begin{align*}
    \mathfrak S_{2413} &= x_1x_2^2 + x_1^2x_2,\textup{ and}\\
    \mathfrak G_{2413} &= x_1x_2^2 + x_1^2x_2 - x_1^2x_2^2.
\end{align*}
\end{exa}

We remark that $\mathfrak S_w$ and $\mathfrak G_w$ can also be computed using many other combinatorial models, such as bumpless pipe dreams \cite{lls21,lls23,weigandt21}, hybrid pipe dreams \cite{ku23}, and vines \cite{bgnst25}. There has also been recent interest in the highest degree part of $\mathfrak G_w$, called the {\color{darkred}\emph{Castelnuovo--Mumford polynomial}} in \cite{psw24}, which we denote $\mathfrak G_w^{\mathrm{top}}$; it can be computed via bumpless vertical-less pipe dreams \cite{cy25}.

\subsection*{Combinatorics of pipe dreams and fireworks permutations}
For any pipe dream $P\in\PD(w)$, we write $P^{(i)}$ for the pipe entering in the $i$-th column and exiting from the $w^{-1}(i)$-th row.

Each tile $T$ of a pipe dream in the region $\{(i,j)\in[n]\times[n]\colon i + j \leq n\}$ involves exactly two pipes entering in the top and right edges of $T$ and exiting from the bottom and left edges of $T$.
\begin{defn}
The {\color{darkred}\emph{primary pipe}} of $T$ is the pipe $P^{(i)}$ that exits from the bottom edge of $T$, and the {\color{darkred}\emph{secondary pipe}} of $T$ is the pipe $P^{(j)}$ that exits from the left edge of $T$.
\end{defn}
\begin{defn}
We say that two pipes $P^{(i)}$ and $P^{(j)}$ {\color{darkred}\emph{cross}} at row $r$ if they are involved in an active cross tile in row $r$.
\end{defn}

\begin{defn}

Let $T$ be a tile of a pipe dream $P$, say in position $(i,j)$. We say that a tile $S\neq T$ appears {\color{darkred}\emph{after}} $T$ if it is in the region $\{(a,b)\colon a > i \textup{ or } a = i, b < j\}$; otherwise we say it appears {\color{darkred}\emph{before}} $T$. The region {\color{darkred}\emph{between}} tiles $T$ and $T'$ is the set of tiles after $T$ and before $T'$, or before $T$ and after $T'$, whichever is nonempty. See Figure~\ref{fig:before-after}. (With this convention, pipes travel ``forward'' as they go from the top edge of the grid to the left edge.)
\end{defn}

\begin{lem}
\label{lem:good-pipes}
Let $P\in\PD(w)$. Let $T$ be a cross tile of $P$ in row $r$ and let $P^{(i)}$ denote the primary pipe of $T$. Then $i$ is not a left-to-right maximum in the string of numbers $w(r)\,w(r+1)\,\dots\,w(n)$. 
\end{lem}
\begin{proof}
Let $P^{(j)}$ denote the secondary pipe of $T$, so that $j> i$. In the region after $T$, pipe $P^{(i)}$ is below $P^{(j)}$. It follows that $j$ appears before $i$ in the string $w(r)\,w(r+1)\,\dots\,w(n)$.
\end{proof}

\begin{figure}[ht]
\begin{center}
\includegraphics{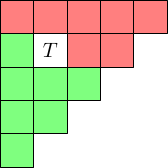}
\end{center}
\caption{Tiles after $T$ are shaded green; tiles before $T$ are shaded red.}
\label{fig:before-after}
\end{figure}

A permutation $w\in S_n$ is called {\color{darkred}\emph{fireworks}} if the initial terms of its decreasing runs are increasing. Equivalently, $w$ is fireworks if there does not exist $i < j$ so that $w(j) < w(j+1) < w(i)$ (i.e., if $w$ is $3$-$12$ avoiding).

The initial terms of decreasing runs of fireworks permutations are exactly the left-to-right maxima of the string of numbers $w(1)\, w(2)\,\dots\,w(n)$.
\subsection*{Diagrams and fireworks permutations}
A {\color{darkred}\emph{diagram}} is defined to be a subset $D\subseteq [n]\times[m]$. We view $D = (D_1, \dots, D_m)$ as a subset of an $n\times m$ grid, where $D_j\colonequals \{i \in[n]\colon (i,j)\in D\}$ encodes the $j$\textsuperscript{th} column: an element $i \in D_j$ corresponds to a box in row $i$ and column $j$.

\begin{defn}
Let $w \in S_n$. The Rothe diagram $D(w)$ is defined to be
\[
D(w) = \{(i,j)\in[n]\times[n]\colon i < w^{-1}(j) \textup{ and } j < w(i)\}.
\]
(See Figure~\ref{fig:31542-rothe}, left. One can construct $D(w)$ pictorially by placing dots in row $i$ and column $w(i)$, then drawing death rays emanating east and south of each dot, and shading in the boxes that remain.)
\end{defn}
\begin{figure}[ht]
\centering
\includegraphics{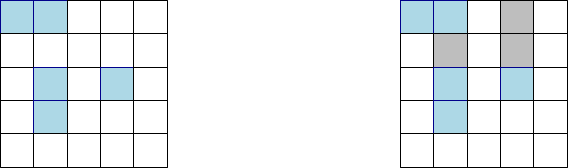}
\caption{The Rothe diagram for $w = 31542$ and its upwards closure.}
\label{fig:31542-rothe}
\end{figure}

\begin{defn}
Let $D$ be a diagram. The {\color{darkred}\emph{upwards closure}} $\overline D$ of $D$ is
\[
\overline D = \{(i,j)\in[n]\times[m]\colon (i',j)\in D\textup{ for some } i' \geq i\},
\]
i.e., the diagram consisting of squares which are above some square in $D$. (See Figure~\ref{fig:31542-rothe}, right.)
\end{defn}

The \definition{weight} $\wt(D)$ of a diagram is the vector whose $i$-th component is the number of boxes of $D$ in row $i$.

\begin{prop}
\label{prop:maximal-weight}
Let $w\in S_n$ be fireworks. Then
\[
\wt(\overline{D(w)})_a = \#\{j>a\colon w(j)\textup{ is not an initial term of a decreasing run}\}.
\]
\end{prop}
\begin{exa}
The permutation $w = 31542$ from Figure~\ref{fig:31542-rothe} is fireworks. In this case,
\[
\{j\colon w(j)\textup{ is not an initial term of a decreasing run}\} = \{2,4,5\};
\]
for $a=1$, $2$, $3$, $4$, $5$, this set has $3$, $2$, $2$, $1$, $0$ elements greater than $a$, respectively.
\end{exa}
\begin{proof}[Proof of Proposition~\ref{prop:maximal-weight}]
By the fireworks assumption on $w$, if $w(i)$ is an initial term of a decreasing run then the $w(i)$-th column of $D(w)$ (and hence of $\overline{D(w)})$ is empty. On the other hand, if $w(i)$ is not an initial term of a decreasing run then $w(i) < w(i-1)$ and so the $w(i)$-th column of $D(w)$ has a box in row $i-1$ (and no boxes in row $i$ or below). It follows that the $w(i)$-th column of $\overline{D(w)}$ is empty if $w(i)$ is the initial term of a decreasing run and has boxes exactly in rows $1, \dots, i-1$ otherwise. The claim follows.
\end{proof}
\begin{cor}
\label{cor:good-pipes-only}
Let $w \in S_n$ be fireworks. For any pipe dream $P$, the componentwise inequality $\wt(P) \leq \wt(\overline{D(w)})$ holds. Furthermore, if $\wt(P)_a < \wt(\overline{D(w)})_a$ then there exists a primary pipe $P^{(i)}$ of a bump tile in row $a$ so that $i$ is not an initial term of a decreasing run.
\end{cor}
\begin{proof}
By Lemma~\ref{lem:good-pipes} there are only
\[
d_a \colonequals \#\{j\colon j>a\textup{ such that } w(j)\textup{ is not an initial term of a decreasing run}\}
\]
many pipes which could feasibly be primary pipes of a cross tile in row $a$. On the other hand, Proposition~\ref{prop:maximal-weight} says that $\wt(\overline{D(w)})_a = d_a$. Thus $\wt(P)_a \leq \wt(\overline{D(w)})_a$, and in order for equality to hold every pipe which is not an initial term of a decreasing run must exit from the bottom edge of row $a$ as the primary pipe of a cross tile.
\end{proof}

\subsection*{Support and M-convexity}
The {\color{darkred}\emph{support}} of a polynomial $f \in \ZZ[x_1, \dots, x_n]$, denoted $\supp(f)$, is defined to be the set $\{\alpha \in \ZZ_{\geq 0}^n \colon \mathbf x^\alpha \textup{ appears in } f\}$ of exponent vectors of monomials appearing with nonzero coefficient in $f$.

\begin{lem}
\label{lem:groth-lower-bounded}
For any $\beta\in\supp(\mathfrak G_w)$, there exists $\alpha\in\supp(\mathfrak S_w)$ satisfying $\alpha \leq \beta$.
\end{lem}
\begin{proof}
Replace the inactive cross tiles of a pipe dream $P\in \PD(w)$ with bump tiles to obtain another pipe dream $P^{\mathrm{red}}\in\PD(w)$ which is necessarily reduced. Then $\wt(P^{\mathrm{red}})\leq \wt(P)$. The claim follows from Theorem~\ref{thm:pd-generates}.
\end{proof}

The {\color{darkred}\emph{Newton polytope}} of $f$, denoted $\Newton(f)$, is the convex hull $\conv(\supp(f))$ of its support. We say $f$ has {\color{darkred}\emph{saturated Newton polytope}}, abbreviated SNP, if $\supp(f) = \Newton(f)\cap\ZZ^n$. Monical, Tokcan, and Yong \cite{mty19} initiated the study of the SNP property and observed its ubiquity in algebraic combinatorics; see their paper for more details.

A set $S\subset\ZZ^n$ is {\color{darkred}\emph{M-convex}} if for any $\alpha,\beta\in S$ and any $i\in[n]$ for which $\alpha_i > \beta_i$, there is an index $j\in[n]$ for which $\alpha_j < \beta_j$ and $\alpha - \mathbf e_i + \mathbf e_j \in S$ and $\beta + \mathbf e_i - \mathbf e_j \in S$. M-convex sets are automatically saturated and are better behaved than saturated sets in general; for example, Lemma~\ref{lem:gp-minkowski} can be interpreted as saying that M-convexity is preserved under Minkowski sum.

There is an interpretation of M-convexity in terms of certain polytopes called generalized permutahedra. A function $z\colon 2^{[n]}\to\RR$ is called {\color{darkred}\emph{submodular}} if
\[
z(I) + z(J) \geq z(I\cup J) + z(I\cap J)\qquad\textup{ for all } I,J\subseteq[n].
\]
\begin{defn}
A {\color{darkred}\emph{generalized permutahedron}} $P\subset\RR^n$ is a polytope defined by
\[
P = \left\{x\in\RR^n\colon \sum_{i\in I}x_i \leq z(I) \textup{ for all } I\subseteq[n]\textup{ and } \sum_{i=1}^n x_i = z([n])\right\}
\]
for some submodular function $z\colon 2^{[n]}\to\RR$ satisfying $z(\emptyset) = 0$.
\end{defn}
A set is M-convex if and only if it is the set of integer points of an integral generalized permutahedron \cite[Theorem 4.15]{murota03}. We refer to \cite{schrijver03} and \cite{postnikov09} for background on generalized permutahedra.
\begin{lem}[{\cite[Thm 44.6, Cor 46.2c]{schrijver03}}]
\label{lem:gp-minkowski}
The Minkowski sum $P_1 + \dots + P_k$ of generalized permutahedra is a generalized permutahedron. Furthermore,
\[
\left(\sum_{j=1}^k P_j\right) \cap \ZZ^n = \sum_{j=1}^k(P_j\cap\ZZ^n),
\]
that is to say, any integer point in a Minkowski sum of generalized permutahedra can be written as a sum of integer points of summands.
\end{lem}

We next discuss a generalized permutahedron called the \emph{Schubert matroid polytope} which will later be used to build more complex generalized permutahedra. 

\begin{defn}
\label{defn:schubert-matroid-polytope}
Given a subset $S \subseteq[n]$ consisting of elements $s_1 < \dots < s_r$, write $\mathcal B$ for the collection of sets $B\subseteq[s_r]$ consisting of elements $b_1 < \dots < b_r$ satisfying $b_i \leq s_i$ for all $i$. Then $\mathcal B$ is the collection of bases of a matroid called the {\color{darkred}\emph{Schubert matroid}}, denoted $\SM(S)$. The indicator vectors $\{\zeta_B\colon B\in\mathcal B\}\subset \ZZ^n$ form an M-convex set, and their convex hull is a generalized permutahedron called the {\color{darkred}\emph{Schubert matroid polytope}}
\[
P(\mathrm{SM}(S))\colonequals\conv\{\zeta_B\colon B\in\mathcal B\}.
\]
See \cite[pg 2]{mty19} or \cite[Thm 10]{fms18} for a description of the rank function of $P(\mathrm{SM}(S))$.
\end{defn}

When $S$ is empty, we set $P(\SM(S)) \colonequals \{0\}$.

Schubert matroid polytopes compute the support of Schubert polynomials.

\begin{thm}[{\cite[Thm 4, Thm 7]{fms18}}]
\label{thm:schub-support}
Let $w \in S_n$, and let $D_1, \dots, D_n$ denote the columns of $D(w)$. Then $\mathfrak S_w$ has M-convex support and
\[
\supp(\mathfrak S_w) = \left(\sum_{i=1}^n P(\SM(D_i))\right)\cap\ZZ^n.
\]
\end{thm}
\subsection*{Lorentzian polynomials and log-concavity}
\begin{defn}[{\cite[Defn 2.6, Thm 2.25]{bh20}}]
A homogeneous polynomial with nonnegative coefficients is \definition{Lorentzian} if:
\begin{itemize}
        \item The support $\supp(f)$ is an M-convex set.
        \item For any $i_1, \dots, i_{d-2} \in [n]$, the Hessian of the quadratic polynomial $q = \frac{\partial}{\partial x_{i_1}} \cdots \frac{\partial}{\partial x_{i_{d-2}}} f$ has at most one positive eigenvalue.
\end{itemize}
\end{defn}
For a polynomial $f = \sum_\alpha c_\alpha \mathbf x^\alpha$, define its \definition{normalization} to be
\[
N(f) = \sum_\alpha c_\alpha \frac{\mathbf x^\alpha}{\alpha!}.
\]
\begin{lem}[{\cite[Lem 4.8]{blp23}}]
\label{lem:bivariate}
Let $f(x_1, \dots, x_n)$ be a polynomial such that $N(f)$ is Lorentzian. Then, the normalization $N(g)$ of the bivariate polynomial $g(q,t)\colonequals f(q,\dots, q, t)$ is also Lorentzian.
\end{lem}
\begin{prop}
\label{prop:log-concavity-slices}
Let $P\subset\RR_{\geq 0}^{n+1}$ be an integral generalized permutahedron. Then the slices
\[
P^{(c)} \colonequals \{x \in \RR^n \colon (x,c) \in P\}, \qquad c \in \ZZ,
\]
are again integral generalized permutahedra, and the numbers
\[
a_c\colonequals \#\{P^{(c)} \cap \ZZ^n\}
\]
form a log-concave sequence.
\end{prop}
\begin{proof}
Consider the \emph{normalized integer point transforms} of $P$ and $P^{(c)}$, defined by
\begin{align*}
\varphi_P(x_1, \dots, x_{n+1}) &\colonequals \sum_{\alpha \in P \cap \ZZ^{n+1}} \frac{\mathbf x^\alpha}{\alpha!},\\
\varphi_{P^{(c)}}(x_1, \dots, x_n) &\colonequals \sum_{\alpha \in P^{(c)} \cap \ZZ^n} \frac{\mathbf x^\alpha}{\alpha!}.
\end{align*}
By \cite[Thm 3.10]{bh20}, $\varphi_P$ is \emph{Lorentzian}. It is known (\cite[Lem 4.4]{blp23}) that $x_i$-degree components of Lorentzian polynomials are again Lorentzian. Then, the first claim follows from the $x_{n+1}$-degree component decomposition
\[
\varphi_P(x_1, \dots, x_{n+1}) = \sum_c \frac{x_{n+1}^c}{c!} \cdot \varphi_{P^{(c)}}(x_1, \dots, x_n).
\]
To prove the second claim, apply Lemma~\ref{lem:bivariate} to the integer point transform
\[
f(x_1, \dots, x_{n+1}) = \sum_{\alpha \in P\cap \ZZ^{n+1}} \mathbf x^\alpha.
\]
The specialization $g(q,t) \colonequals f(q,\dots,q,t)$ is given by
\[
g(q,t) = \sum_c a_c \cdot t^c q^{\deg(f) - c}.
\]
Because $N(g)$ is Lorentzian, \cite[Ex 2.26]{bh20} yields log-concavity of $a_c$.
\end{proof}

\section{Construction of pipe dreams of given weight}

In this section we prove Theorem~\ref{thm:main}. The key tool is a pipe dream construction (Theorem~\ref{thm:main2-steps}) which can be iterated to produce pipe dreams of higher weight. We begin with several preparatory lemmas about the local structure of pipes within a pipe dream. 

\begin{lem}
\label{lem:exists-bump}
Suppose there exists a bump tile $T$ in row $r$ whose primary pipe $P^{(i)}$ is also the primary pipe of an active cross tile below $T$. Let $T'$ denote the highest such cross tile, say in row $r'$, and let $P^{(j)}$ denote the secondary pipe of $T'$.
\begin{enumerate}
\item The pipe $P^{(j)}$ is not the primary pipe of any crossing tile, active or inactive, in rows $[r, r')$.

\item If $S$ is any tile in rows $[r,r')$ with primary pipe $P^{(j)}$ and secondary pipe $P^{(k)}$, then $j>k$ and $P^{(k)}$ crosses $P^{(j)}$ before tile $T'$.
\end{enumerate}
\end{lem}
See Figure~\ref{fig:exists-bump} for cartoons depicting parts (1) and (2) of Lemma~\ref{lem:exists-bump}.
\begin{figure}[ht]
\begin{center}
\includegraphics{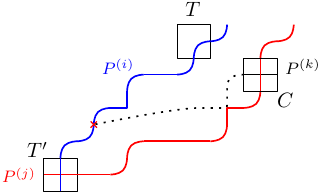}\qquad\qquad \includegraphics{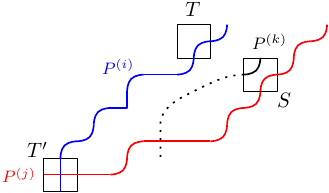}
\end{center}
\caption{Left, Lemma~\ref{lem:exists-bump}(1): $P^{(k)}$ cannot cross $P^{(i)}$ or $P^{(j)}$; it has nowhere to go.\\ Right, Lemma~\ref{lem:exists-bump}(2): $P^{(k)}$ cannot cross $P^{(i)}$, so it must cross $P^{(j)}$ before $T'$.}
\label{fig:exists-bump}
\end{figure}
\begin{proof}
Our argument loosely follows Figure~\ref{fig:exists-bump}.

\textbf{Part (1).} Suppose that $P^{(j)}$ is the primary pipe to a cross tile $C$ in row $s \in [r,r')$, and write $P^{(k)}$ for the secondary pipe of $C$. (Pipes $P^{(i)}$ and $P^{(k)}$ are distinct as $P^{(i)}$ first crosses $P^{(j)}$ at row $r'$, whereas $P^{(k)}$ first crosses $P^{(j)}$ at row $s < r'$.) As $P^{(j)}$ and $P^{(k)}$ cross at $C$, the pipe $P^{(k)}$ must stay weakly above $P^{(j)}$ after tile $C$. As $P^{(i)}$ is weakly below $P^{(j)}$ after $T'$, pipes $P^{(i)}$ and $P^{(k)}$ must cross after $C$ but before $T'$, contradicting minimality of $r'$.


\textbf{Part (2).} Let $S$ be a (necessarily bump) tile in rows $[r,r')$ with primary pipe $P^{(j)}$ and secondary pipe $P^{(k)}$. By the minimality assumption on $r'$, the pipe $P^{(k)}$ cannot cross $P^{(i)}$ in rows $[r,r')$, so $P^{(k)}$ must cross $P^{(j)}$ as the primary pipe of some cross tile that is before $T'$.
\end{proof}

\begin{lem}
\label{lem:three-in-a-tile}
Suppose there exists a bump tile $T$ in row $r$ whose primary pipe $P^{(i)}$ and secondary pipe $P^{(j)}$ do not cross. Let $P^{(k)}$ be a pipe which exits from the left between $P^{(i)}$ and $P^{(j)}$. Then $P^{(k)}$ crosses $P^{(i)}$ in some row $s > r$, or $P^{(k)}$ crosses $P^{(j)}$ in some row $s \geq r$.
\end{lem}
\begin{figure}[ht]
\begin{center}
\includegraphics{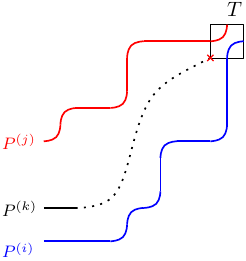}
\end{center}
\caption{$P^{(k)}$ cannot be below $P^{(j)}$ and above $P^{(i)}$ in rows $\geq r$, because $P^{(i)}$ and $P^{(j)}$ are too close together at $T$.}
\label{fig:three-in-a-tile}
\end{figure}

\begin{proof}
If $P^{(k)}$ does not cross $P^{(i)}$ in some row $s > r$, then $P^{(k)}$ is above $P^{(i)}$ in rows $\geq r$. Similarly, if $P^{(k)}$ does not cross $P^{(j)}$ in some row $s \geq r$, then $P^{(k)}$ is below $P^{(j)}$ in rows $\geq r$. These two conditions on $P^{(k)}$ are contradictory to the existence of a bump tile $T$ in row $r$ involving $P^{(i)}$ and $P^{(j)}$. See Figure~\ref{fig:three-in-a-tile}.
\end{proof}

\begin{lem}[cf.\ {\cite[Lem 5.2]{cy25}}]
\label{lem:case-2-propagates}
Suppose that $P^{(i)}$ and $P^{(j)}$ are the two pipes involved in tiles $T$ and $T'$ with $P^{(i)}$ always above $P^{(j)}$ between $T$ and $T'$. Let $R$ denote the region enclosed by $P^{(i)}$ and $P^{(j)}$ between $T$ and $T'$. Then any pipe $P^{(k)}$ $(k \notin \{i,j\})$ appearing in $R$ must cross both $P^{(i)}$ and $P^{(j)}$ between $T$ and $T'$.
    
In this case, $P^{(k)}$ is the primary pipe of its active crossing with $P^{(i)}$ if and only if $P^{(k)}$ is the primary pipe of its active crossing with $P^{(j)}$.
\end{lem}

\begin{figure}[ht]
\begin{center}\includegraphics{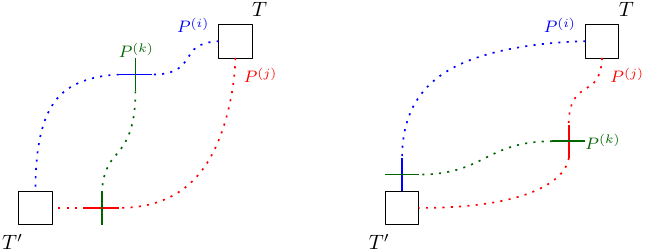}\end{center}
\caption{Cartoon for Lemma~\ref{lem:case-2-propagates}.}
\label{fig:case-2-propagates}
\end{figure}

\begin{proof}
Any pipe $P^{(k)}$ entering or exiting $R$ must cross pipe $P^{(i)}$ or $P^{(j)}$; thus if $P^{(k)}$ crosses one of $P^{(i)}$ or $P^{(j)}$ between $T$ and $T'$ it must cross the other.

Furthermore, if $P^{(k)}$ crosses either $P^{(i)}$ or $P^{(j)}$ as the primary pipe as it enters $R$, then $P^{(k)}$ is above both pipes $P^{(i)}$ and $P^{(j)}$ before entering $R$; hence it must be below both pipes $P^{(i)}$ and $P^{(j)}$ after exiting $R$ and must cross both $P^{(i)}$ and $P^{(j)}$ as the primary pipe.
\end{proof}

\newtheorem*{thm:main2-steps}{Theorem~\ref{thm:main2-steps}}
\begin{thm:main2-steps}
Let $w\in S_n$ be a fireworks permutation and let $P \in \PD(w)$ be a pipe dream with $\wt(P) < \wt(\overline{D(w)})$. For any $a$ with $\wt(P)_a < \wt(\overline{D(w)})_a$, there exists a pipe dream $Q \in \PD(w)$ satisfying:
\begin{itemize}
\item $\wt(Q)_b = \wt(P)_b$ for $b<a$
\item $\wt(Q)_a = \wt(P)_a + 1$
\item $\wt(Q)_b \leq \wt(P)_b$ for $b>a$.
\end{itemize}
\end{thm:main2-steps}

\begin{proof}
As $\wt(P)_a < \wt(\overline{D(w)})_a$, Corollary~\ref{cor:good-pipes-only} implies that there exists a primary pipe $P^{(i)}$ of a bump tile $T$ in row $a$ where $i$ is not an initial term of a decreasing run. (Our goal is to replace $T$ with a cross tile without changing the associated permutation, possibly at the cost of replacing some cross tiles in rows $>a$ with bumps; cf.\ the introduction.)

Let $P^{(j)}$ denote the secondary pipe of $T$.

\textbf{Case 0.} Assume that $i<j$. In this case, replace $T$ with a cross tile to obtain a pipe dream $Q$. As the label of pipe $P^{(i)}$ is smaller than that of pipe $P^{(j)}$, this new cross tile is an inactive crossing, so $Q\in\PD(w)$ is the desired pipe dream.

\textbf{Case 1.} Assume that $i>j$ and $P^{(i)}$ is not the primary pipe of any active cross tile below row $a$. We first claim that $w^{-1}(j) > w^{-1}(i)$, i.e., that $P^{(i)}$ and $P^{(j)}$ cross in $P$ (necessarily below row $a$ with $P^{(j)}$ as the primary pipe). Otherwise, if $P^{(i)}$ and $P^{(j)}$ do not cross, then $w^{-1}(j) < w^{-1}(i)$; the fireworks assumption on $w$, combined with the fact that $i$ is not a left-to-right maximum, implies that the number $w(w^{-1}(i) - 1)\equalscolon k$ immediately to the left of $i$ in the string $w(1)\,w(2)\,\dots\,w(n)$ is larger than $i$. Hence both $w^{-1}(j) < w^{-1}(k)$ and $j < k$ hold, i.e., that the pipes $P^{(j)}$ and $P^{(k)}$ do not cross. Lemma~\ref{lem:three-in-a-tile} then implies that $P^{(i)}$ and $P^{(k)}$ cross below row $a$ with $P^{(i)}$ as the primary pipe, contradictory to assumption on $P^{(i)}$. Hence, $P^{(i)}$ and $P^{(j)}$ cross.

Let $T'$ denote the active crossing tile where $P^{(i)}$ and $P^{(j)}$ cross. Construct a pipe dream $P'$ by:
\begin{itemize}
\item Replacing $T$ with a cross tile, and
\item Replacing all inactive cross tiles involving $P^{(i)}$ or $P^{(j)}$ between tiles $T$ and $T'$ with bump tiles.
\end{itemize}
Note that the cross tile $T'$ is necessarily inactive. See Figure~\ref{fig:case-2}.

\begin{figure}[ht]
\begin{center}\includegraphics{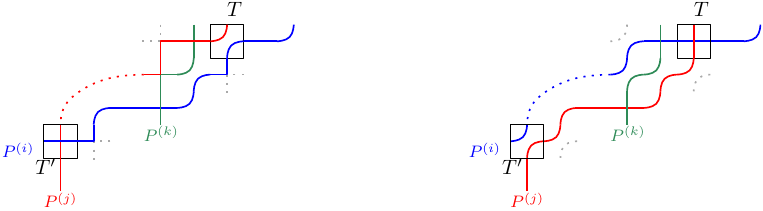}\end{center}
\caption{Construction for Case 1 of Theorem~\ref{thm:main2-steps}}
\label{fig:case-2}
\end{figure}
By Lemma~\ref{lem:case-2-propagates}, any pipe $P^{(k)}$ crossing either $P^{(i)}$ or $P^{(j)}$ between $T$ and $T'$ crosses both $P^{(i)}$ and $P^{(j)}$ between $T$ and $T'$; it follows that $P' \in \PD(w)$.

The weight of $P'$ satisfies:
\begin{itemize}
\item $\wt(P')_b = \wt(P)_b$ for $b < a$,
\item $\wt(P')_a \in \{\wt(P)_a, \wt(P)_a + 1\}$, depending on whether or not $P^{(j)}$ is the primary pipe of an inactive crossing in row $a$, and
\item $\wt(P')_b \leq \wt(P)_b$ for $b > a$.
\end{itemize}

If $\wt(P')_a = \wt(P)_a + 1$ then $Q = P'$ is the desired pipe dream. In case $\wt(P')_a = \wt(P)_a$, recurse: the pipe $(P')^{(i)}$ of $P'$ is again not the primary pipe of any active cross tile below row $a$ because its path in $P'$ agrees with the path of $P^{(i)}$ in $P$ after $T'$ and its path in $P'$ has no active crossings between $T$ and $T'$ by construction. For as long as the index $i$ of $(P')^{(i)}$ remains larger than the index $j'$ of the secondary pipe $(P')^{(j')}$ of the bump tile $T_{P'}$ in row $a$ for which $(P')^{(i)}$ is primary, repeatedly apply the above construction; if $i$ is ever smaller than $j'$ then the assumption of Case 0 holds and replacing $T_{P'}$ with an inactive cross tile gives the desired pipe dream $Q \in \PD(w)$. This procedure terminates after finitely many steps because the construction in Case 1 moves the tile $T$ strictly to the left.

\textbf{Case 2.} Assume that $P^{(i)}$ is the primary pipe of an active cross tile below row $a$.

Let $T'$ denote the highest such cross tile and let $P^{(\ell)}$ denote the secondary pipe of $T'$. The pipe $P^{(\ell)}$ is primary for some tile $S$ in row $a$. Part (1) of Lemma~\ref{lem:exists-bump} implies that $S$ is a bump tile, and part (2) of Lemma~\ref{lem:exists-bump} implies that the secondary pipe $P^{(\ell')}$ of $S$ crosses $P^{(\ell)}$ at some tile $S'$ before $T'$. (See Figure~\ref{fig:case-1}, left.)

Construct a pipe dream $P'$ by:
\begin{itemize}
\item Replacing $S$ with a cross tile, and
\item Replacing all inactive cross tiles involving $P^{(\ell)}$ or $P^{(\ell')}$ between tiles $S$ and $S'$ with bump tiles.
\end{itemize}
Note that the cross tile $S'$ is necessarily inactive. See Figure~\ref{fig:case-1}, right.
\begin{figure}[ht]
\begin{center}
\includegraphics{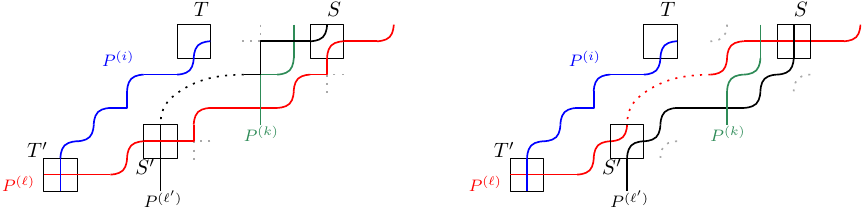}
\end{center}
\caption{Construction for Case 2 of Theorem~\ref{thm:main2-steps}}
\label{fig:case-1}
\end{figure}

By Lemma~\ref{lem:case-2-propagates}, any pipe $P^{(k)}$ crossing either $P^{(\ell)}$ or $P^{(\ell')}$ between $S$ and $S'$ crosses both $P^{(\ell)}$ and $P^{(\ell')}$ between $S$ and $S'$; it follows that $P' \in \PD(w)$.
\begin{itemize}
\item $\wt(P')_b = \wt(P)_b$ for $b < a$,
\item $\wt(P')_a \in\{\wt(P)_a, \wt(P)_a+1\}$, depending on whether or not $P^{(\ell')}$ is the primary pipe of an inactive crossing in row $a$, and
\item $\wt(P')_b \leq \wt(P)_b$ for $b > a$.
\end{itemize}
If $\wt(P')_a = \wt(P)_a + 1$ then $Q = P'$ is the desired pipe dream. In case $\wt(P')_a = \wt(P)_a$, recurse: the pipe $(P')^{(i)}$ of $P'$ is again the primary pipe of an active cross tile below row $a$. Now repeatedly apply the above construction: after each step, the tile in row $a$ where $P^{(\ell)}$ is primary moves strictly closer to $T$; if the tile of $P$ in row $a$ where $P^{(\ell)}$ is primary is adjacent to $T$ then $P^{(\ell')} = P^{(i)}$, in which case $P^{(\ell')}$ is not the primary pipe of any inactive crossing in row $a$, so that $\wt(P')_a = \wt(P)_a + 1$, and $P'$ is the desired pipe dream.
\end{proof}
\begin{rem}
Theorem~\ref{thm:main2-steps} gives a new proof of the fact \cite[Thm 3.15]{mss25} that $\mathbf x^{\wt(\overline{D(w)})}$ appears in the Grothendieck polynomial $\mathfrak G_w$ of a fireworks permutation.
\end{rem}
\begin{rem}
Our construction in the proof of Theorem~\ref{thm:main2-steps} yields pipe dreams $Q$ which are equal to $P$ in rows $<a$.
\end{rem}
\newtheorem*{thm:main}{Theorem~\ref{thm:main}}
\begin{thm:main}
Let $w$ be a fireworks permutation. If $\mathbf x^\alpha$ appears in $\mathfrak G_w$ and $x_i\cdot \mathbf x^\alpha$ divides $\mathbf x^{\wt(\overline{D(w)})}$, then $x_i\cdot \mathbf x^\alpha$ appears in $\mathfrak G_w$. In particular,
\[
\supp(\mathfrak G_w) = \bigcup_{\alpha\in\supp(\mathfrak S_w)}[\alpha,\wt(\overline{D(w)})].
\]
\end{thm:main}
\begin{proof}
Suppose that $\mathbf x^\alpha$ appears in $\mathfrak G_w$ and that $x_i\cdot \mathbf x^\alpha$ divides $\mathbf x^{\wt(\overline{D(w)})}$. By Theorem~\ref{thm:pd-generates}, there is a pipe dream $P\in\PD(w)$ with $\wt(P) = \alpha$. By repeated application of Theorem~\ref{thm:main2-steps}, there exists a pipe dream $Q\in\PD(w)$ with $\mathbf x^{\wt(Q)} = x_i\cdot\mathbf x^\alpha$; Theorem~\ref{thm:pd-generates} guarantees that $x_i\cdot \mathbf x^\alpha$ appears in $\mathfrak G_w$. The claim follows.
\end{proof}
\section{M-convexity of $\supp(\mathfrak G_w)$}
In this section, we prove Corollary~\ref{cor:m-convex} by observing that the support of $\mathfrak G_w$ can be decomposed as a Minkowski sum of smaller M-convex sets (Proposition~\ref{prop:psp-integer-points}).
\begin{defn}
Given a subset $S\subseteq[n]$ consisting of elements $s_1 < \dots < s_r$, write $\mathcal B_{\mathrm{sp}}$ for the {\color{darkred}\emph{spanning sets}} of the Schubert matroid $\SM(S)$; it is the collection of sets $B'\subseteq[s_r]$ consisting of elements $b_1 < \dots < b_{r'}$, with $r' \geq r$, satisfying $b_i \leq s_i$ for all $i\in[r]$.
\end{defn}
A result of Fujishige \cite{fujishige84} implies that the homogenizations $\widetilde\zeta_{B'}\colonequals (\zeta_{B'}, s_r - |B'|)\in\ZZ^{n+1}$ form an M-convex set.
\begin{defn}[cf.\ Definition~\ref{defn:schubert-matroid-polytope}]
\label{defn:schubert-spanning-set-polytope}
The convex hull
\[
\widetilde P_{\mathrm{sp}}(\mathrm{SM}(S))\colonequals\conv\{\widetilde\zeta_{B'}\colon B'\in\mathcal B_{\mathrm{sp}}\}.
\]
is the {\color{darkred}\emph{homogenized Schubert spanning set polytope}}. See \cite[Prop 5.6]{hmss24} for a description of the rank function.

The projection $\varphi\colon \RR^{n+1}\to\RR^n$ forgetting the last coordinate yields an integral equivalence between $\widetilde P_{\mathrm{sp}}(\mathrm{SM}(S))$ and its image $P_{\mathrm{sp}}(\mathrm{SM}(S))$, called the {\color{darkred}\emph{Schubert spanning set polytope}}. (The set $P_{\mathrm{sp}}(\mathrm{SM}(S))\cap\ZZ^n$ of integer points is \emph{M${}^\natural$-convex}; see \cite[pg 18]{bh20} or \cite[\S 4.7]{murota03}.)
\end{defn}
As with Schubert matroid polytopes, when $S$ is empty, we set $\widetilde P_{\mathrm{sp}}(\SM(S)) = P_{\mathrm{sp}}(\SM(S)) \colonequals \{0\}$.

\begin{lem}
\label{lem:homogenization-trick}
For subsets $D_1, \dots, D_m \subseteq[n]$, the Minkowski sum
\[
\sum_{j=1}^m \widetilde P_{\mathrm{sp}}(\SM(D_j))
\]
is a generalized permutahedron; in particular,
\[
\left(\sum_{j=1}^mP_{\mathrm{sp}}(\SM(D_j))\right) \cap \ZZ^n = \sum_{j=1}^m \left(P_{\mathrm{sp}}(\SM(D_j)) \cap \ZZ^n\right).
\]
\end{lem}
\begin{proof}
As $\widetilde P_{\mathrm{sp}}(\SM(D_j))$ is a generalized permutahedron, the claim follows from Lemma~\ref{lem:gp-minkowski}, combined with the fact that (de-)homogenization commutes with taking Minkowski sums.
\end{proof}
\begin{prop}[{cf.\ \cite[Thm 3.6]{mss25}}]
\label{prop:psp-integer-points}
Fix subsets $D_1, \dots, D_m \subseteq[n]$ and set $P\colonequals P(\SM(D_1)) + \dots + P(\SM(D_m))$. Write $D = (D_1, \dots, D_m)$ for the diagram whose columns are $D_i$. Then
\[
\left(\sum_{j=1}^mP_{\mathrm{sp}}(\SM(D_j))\right) \cap \ZZ^n = \bigcup_{\beta\in P\cap\ZZ^n}[\beta,\wt(\overline D)].
\]
\end{prop}
\begin{proof}
Take any 
\[
\alpha\in \left(\sum_{j=1}^mP_{\mathrm{sp}}(\SM(D_j))\right) \cap \ZZ^n
\]
and using Lemma~\ref{lem:homogenization-trick} write $\alpha = \alpha^{(1)} + \dots + \alpha^{(m)}$ for $\alpha^{(j)} \in P_{\mathrm{sp}}(\SM(D_j)) \cap \ZZ^n$. By definition, $\alpha^{(j)} = \zeta_{B'}$ for some spanning set $B'$ of $\SM(D_j)$. Take a basis $B\subseteq B'$ and set $\beta^{(j)} \colonequals \zeta_B \in P(\SM(D_j))$. Writing $d_j \colonequals \max\{a\in D_j\}$, note that $\beta^{(j)}\leq \alpha^{(j)} \leq \omega_{d_j}$, where $\omega_{d_j} = (1,\dots,1,0,\dots,0)$ is the vector whose first $d_j$ coordinates are $1$ and the rest are $0$. The sum $\beta\colonequals \beta^{(1)} + \dots + \beta^{(m)}$ is in $\supp(P)$, and the sum $\omega_{d_1} + \dots + \omega_{d_m}$ is $\wt(\overline D)$.

It follows that
\[
\alpha^{(1)} + \dots + \alpha^{(m)} = \alpha \in [\beta,\wt(\overline D)].
\]
The converse is essentially proven in \cite[proof of Thm 3.6]{mss25}; we paraphrase the argument here for the reader's convenience. Take any
\[
\alpha\in\bigcup_{\beta\in P\cap\ZZ^n}[\beta,\wt(\overline D)]
\]
and, using Lemma~\ref{lem:gp-minkowski}, choose a decomposition $\beta = \beta^{(1)} + \dots + \beta^{(m)}$ with $\beta^{(i)}\in P(\SM(D_i))$. For each $j$, there are exactly $\beta_j$ indices $i$ such that $\beta^{(i)}_j = 1$, and the rest have $\beta^{(i)}_j = 0$. Define the set
\[
T_j \colonequals \{i \colon \beta^{(i)}_j = 0 \textup{ and } j \in \overline D_i\},
\]
which has $\wt(\overline D)_j - \beta_j$ elements. Pick an arbitrary subset $S_j\subset T_j$ of size $\alpha_j - \beta_j$; such a subset $S_j$ exists since $\wt(\overline D)_j - \beta_j \geq \alpha_j - \beta_j$. Then the vector $\alpha^{(i)}$, defined coordinatewise as $\alpha^{(i)}_j \colonequals \beta^{(i)}_j + \mathbf 1_{\{i \in S_j\}}$ is in $P_{\mathrm{sp}}(\SM(D_i))$. By construction,
\[
\alpha = \alpha^{(1)} + \dots + \alpha^{(m)} \in \sum_{i=1}^m P_{\mathrm{sp}}(\SM(D_i)).\qedhere
\]
\end{proof}

\newtheorem*{cor:m-convex}{Corollary~\ref{cor:m-convex}}
\begin{cor:m-convex}
When $w$ is fireworks, the homogenized support $\supp(\widetilde{\mathfrak G}_w)$ is M-convex, i.e., $\widetilde{\mathfrak G}_w$ has SNP and its Newton polytope is a generalized permutahedron. In particular, $\mathfrak G_w$ has SNP, each degree component $\mathfrak G_w^{(c)}$ also has M-convex support, and the sizes $|\supp(\mathfrak G_w^{(c)})|$ of their supports form a log-concave sequence.
\end{cor:m-convex}
\begin{proof}
By Theorem~\ref{thm:main}, Theorem~\ref{thm:schub-support}, and Proposition~\ref{prop:psp-integer-points}, there are equalities
\[
\supp(\mathfrak G_w) = \bigcup_{\alpha\in\supp(\mathfrak S_w)}[\alpha,\wt(\overline{D(w)})] = \left(\sum_{j=1}^nP_{\mathrm{sp}}(\SM(D_j))\right) \cap \ZZ^n,
\]
where $D_j$ denotes the $j$-th column of $D(w)$.

Homogenizing the above equality gives
\[
\supp(\widetilde{\mathfrak G}_w) = \left(\sum_{j=1}^n\widetilde P_{\mathrm{sp}}(\SM(D_j))\right) \cap \ZZ^{n+1},
\]
so $\widetilde{\mathfrak G}_w$ has M-convex support by Lemma~\ref{lem:homogenization-trick}. The remaining two claims follow from Proposition~\ref{prop:log-concavity-slices}.
\end{proof}
\begin{rem}[{\cite[Thm 3.6]{mss25}}]
In general, Theorem~\ref{thm:schub-support} and Proposition~\ref{prop:psp-integer-points} together imply
\[
\supp(\mathfrak G_w)\subseteq \left(\sum_{j=1}^nP_{\mathrm{sp}}(\SM(D_j))\right)\cap\ZZ^n.
\]
This inclusion is often strict. For example, if $\mathfrak G_w^{\mathrm{top}}$ is not a scalar multiple of a monomial then the inequality is strict: observe that $\wt(\overline{D(w)}) \not\in\supp(\mathfrak G_w)$, as all monomial\emph{s} in $\mathfrak G_w^{\mathrm{top}}$ must (necessarily properly) divide the monomial $\mathbf x^{\wt({\overline{D(w)}})}$.

On the other hand, when $\mathfrak G_w^{\mathrm{top}}$ is a single monomial, one can prove that $\mathfrak G_w^{\mathrm{top}} = c\cdot \mathbf x^{\wt(\overline{D(w)})}$, and we expect the inclusion to be an equality in this case.
\end{rem}
\begin{conj}[cf.~\eqref{eqn:master-support}]
The inclusion
\[
\supp(\mathfrak G_w)\subseteq \left(\sum_{j=1}^nP_{\mathrm{sp}}(\SM(D_j))\right)\cap\ZZ^n
\]
is an equality if and only if $\mathfrak G_w^{\mathrm{top}}$ is a scalar multiple of a monomial.
\end{conj}
\section{Grothendieck polynomials with inclusion-maximal support}
In this section, we prove Corollary~\ref{cor:maximal-newton} using the support formula of Theorem~\ref{thm:main} to reduce the problem to a polytopal containment computation (Lemma~\ref{lem:psp-inclusion}).
\begin{defn}
Given integers $b_1, \dots, b_m$, the {\color{darkred}\emph{layered permutation}} with \emph{block sizes} $(b_1, \dots, b_m)$ is the longest element of the Young subgroup $S_{b_1}\times S_{b_2}\times\dots\times S_{b_m}$; i.e., it is the permutation with one-line notation
\[
c_1 \, (c_1 - 1)\,\dots \,1\, c_2\, (c_2 - 1)\,\dots\,(c_1 + 1)\, \dots\, c_m\, (c_m - 1)\,\dots\,(c_{m-1}+1),
\]
where $c_i \colonequals b_1 + \dots + b_i$.
\end{defn}
A permutation is layered if and only if it is $231$- and $312$-avoiding, and every layered permutation is fireworks.

\begin{lem}
\label{lem:l2r-max-empty}
The $w(i)$-th column $D(w)_{w(i)}$ of $D(w)$ is empty if and only if $w(i)$ is a left-to-right maximum of $w(1)\,w(2)\,\dots\,w(n)$.
\end{lem}
\begin{proof}
The $w(i)$-th column of $D(w)$ is empty if and only if the first $i-1$ rows of $D(w)_{w(i)}$ are empty, or equivalently if $w(i) > w(j)$ for all $j \leq i-1$.
\end{proof}
\begin{lem}
\label{lem:fireworks-layered-tips}
Let $w(i)$ be a non-left-to-right maximum of $w(1)\,w(2)\,\dots\,w(n)$ and let $D(w)_{w(i)}$ denote the $w(i)$-th column of $D(w)$. Let $w(a)$ be the initial term of the descending run of $w(1)\,w(2)\,\dots\,w(n)$ containing $w(i)$.
\begin{itemize}
\item If $w$ is fireworks, then $a \leq i-1$, and $D(w)_{w(i)}\supseteq [a,i-1]$, and $\max(D(w)_{w(i)}) = i-1$. 
\item If $w$ is layered, then $D(w)_{w(i)} = [a,i-1]$.
\end{itemize}
\end{lem}
\begin{proof}
Assume first that $w$ is fireworks. As $w(i)$ is not a left-to-right maximum of $w(1)\,w(2)\,\dots\,w(n)$, it is not an initial term of a decreasing run, i.e., $a \leq i-1$. Since $w(a)\,w(a+1)\,\dots\,w(i)$ forms a decreasing run, the inclusion $[a,i-1]\subseteq D(w)_{w(i)}$ holds. The first item then follows from the fact that $D(w)_{w(i)}$ can have boxes only in rows $\leq i-1$.

If $w$ is further assumed to be a layered permutation, then $w(a)$ is the initial term of a block of $w$ and this block contains $w(i)$. Then $w(j) < w(i)$ for all $j \leq a-1$, so $D(w)_{w(i)}$ does not have any boxes in rows $\leq a-1$. The previous item now implies $D(w)_{w(i)} = [a,i-1]$.
\end{proof}

\begin{lem}
\label{lem:tip-inclusion}
Let $w$ be a fireworks permutation and write $\pi(w)$ for the layered permutation whose block sizes are the lengths of descending runs of $w$. Then $D(w)_{w(i)} \supseteq D(\pi(w))_{(\pi(w))(i)}$.
\end{lem}
\begin{proof}
Since $w$ and $\pi(w)$ are both fireworks with descending runs of the same length, $w(i)$ is a left-to-right maximum of $w(1)\,w(2)\,\dots\,w(n)$ if and only if $(\pi(w))(i)$ is a left-to-right maximum of $(\pi(w))(1)\,(\pi(w))(2)\,\dots\,(\pi(w))(n)$. For such $i$, Lemma~\ref{lem:l2r-max-empty} gives $D(w)_{w(i)} = D(\pi(w))_{(\pi(w))(i)} = \emptyset$. For all other $i$, Lemma~\ref{lem:fireworks-layered-tips} gives 
\[
D(\pi(w))_{(\pi(w))(i)} = [a,i-1]\subseteq D(w)_{w(i)}.\qedhere
\]
\end{proof}
\begin{rem}
The assignment $w\mapsto \pi(w)$ has appeared before in \cite[Defn 6.1]{psw24}. They show that, for any permutation $w$, there exists a unique inverse fireworks permutation $\pi(w)$ such that $\mathfrak G_w^{\mathrm{top}} = c_w \cdot \mathfrak G_{\pi(w)}^{\mathrm{top}}$ for some $c_w \in \ZZ_{>0}$. When $w$ is fireworks, the corresponding inverse fireworks permutation is the layered permutation of Lemma~\ref{lem:tip-inclusion}.
\end{rem}
\begin{lem}
\label{lem:psp-inclusion}
Let $A\subseteq B\subseteq[n]$ be two nonempty sets satisfying $\max(A) = \max(B)$. Then $P_{\mathrm{sp}}(\SM(B)) \subseteq P_{\mathrm{sp}}(\SM(A))$.
\end{lem}
\begin{proof}
Write $k\colonequals \max(A) = \max(B)$ and $\omega_k\colonequals (1, \dots, 1, 0,\dots,0)$ for the vector whose first $k$ coordinates are $1$ and the rest are $0$. From the equalities
\begin{align*}
P_{\mathrm{sp}}(\SM(A))\cap\ZZ^n &= \bigcup_{\alpha\in P(\SM(A))\cap\ZZ^n}[\alpha,\omega_k],\\
P_{\mathrm{sp}}(\SM(B))\cap\ZZ^n &= \bigcup_{\beta\in P(\SM(B))\cap\ZZ^n}[\beta,\omega_k],
\end{align*}
it suffices to show that for any $\beta\in P(\SM(B))\cap\ZZ^n$ there exists $\alpha\in P(\SM(A))\cap\ZZ^n$ such that $\beta\geq\alpha$. To this end, let $a_1 < \dots < a_{|A|}$ and $b_1 < \dots < b_{|B|}$ denote the elements of $A\subseteq B$. By construction (see Definition~\ref{defn:schubert-matroid-polytope}), any $\beta \in P(\SM(B))\cap \ZZ^n$ is an indicator function $\beta = \mathbf 1_S$ of a set $S\subseteq[n]$ consisting of elements $s_1 < \dots < s_{|B|}$ satisfying $s_i \leq b_i$. The inclusion $A\subseteq B$ implies that $b_i \leq a_i$ for all $i\leq |A|$, so that the set $T = \{s_1 < \dots < s_{|A|}\}$ satisfies the conditions required in Definition~\ref{defn:schubert-matroid-polytope} so that $\alpha \colonequals \mathbf 1_T\in P(\SM(A))$. This vector $\alpha$ is the desired vector.
\end{proof}

\newtheorem*{cor:maximal-newton}{Corollary~\ref{cor:maximal-newton}}

\begin{cor:maximal-newton}
Let $w$ be a fireworks permutation and write $\pi(w)$ for the layered permutation whose block sizes are the lengths of descending runs of $w$. Then $\supp(\mathfrak G_{\pi(w)})\supseteq\supp(\mathfrak G_w)$.
\end{cor:maximal-newton}
\begin{proof}
Write $D_i\colonequals D(w)_{w(i)}$ and $D_i'\colonequals D(\pi(w))_{(\pi(w))(i)}$. Lemma~\ref{lem:tip-inclusion} implies that $D_i \supseteq D_i'$ and $\max(D_i) = \max(D_i')$, so Lemma~\ref{lem:psp-inclusion} gives
\[
\supp(\mathfrak G_w) = \left(\sum_{i=1}^n P_{\mathrm{sp}}(\SM(D_i))\right)\cap\ZZ^n \subseteq \left(\sum_{i=1}^n P_{\mathrm{sp}}(\SM(D_i'))\right)\cap\ZZ^n = \supp(\mathfrak G_{\pi(w)}).\qedhere
\]
\end{proof}

Starting with Gao's paper \cite{gao21} (cf.\ especially \cite[Conj 3.2]{gao21}, independently also observed by Christian Gaetz), there has been a recent surge of interest in the study of the interplay between pattern containment, layered permutations, and maximization problems in Schubert calculus \cite{mt22, gl24, dennin25, zhang25}.

In light of Corollary~\ref{cor:maximal-newton}, and the conjecture \cite[Conj 5.2]{gl24} that $|\supp(\mathfrak S_w)|$ is maximized at a layered permutation, it is natural to conjecture:
\begin{conj}
\label{conj:groth-supp-maximized}
For fixed $n$, the maximum $\max_{w\in S_n}|\supp(\mathfrak G_w)|$ is achieved at a layered permutation.
\end{conj}

We remark that the recent counterexample in \cite{app26}, which appeared after the initial version of this paper, reduces our conviction in Conjecture~\ref{conj:groth-supp-maximized}. We have verified Conjecture~\ref{conj:groth-supp-maximized} up to $n = 8$. Our computations for the maximal values of $|\supp(\mathfrak G_w)|$ are recorded in Table~\ref{tab:data}. Let $\mathcal L_n$ denote the set of layered permutations in $S_n$. Among $w\in\mathcal L_9$, the permutation $w = 143298765$ maximizes $|\supp(\mathfrak G_w)|$.
\begin{center}
\begin{table}[ht]
\begin{tabular}{c|c|c}
$n$ & $\max_{w\in S_n}|\supp(\mathfrak G_w)|$ & $w\in S_n$ achieving maximum\\\hline
$1$ & $1$ & $1$\\
$2$ & $1$ & $12$\\
$3$ & $3$ & $132$\\
$4$ & $9$ & $1432$\\
$5$ & $40$ & $12543$\\
$6$ & $236$ & $132654$\\
$7$ & $1808$ & $1327654$\\
$8$ & $14087$ & $13287654$\\\hline \hline
$n$ &$\max_{w\in\mathcal L_n}|\supp(\mathfrak G_w)|$ & $w\in\mathcal L_n$ achieving maximum\\\hline
$9$ & $129318$ & $143298765$\\
\end{tabular}
\vspace{-2ex}
\caption{Maximal values for $|\supp(\mathfrak G_w)|$ over $w\in S_n$ with $n\leq 8$, and over $w\in\mathcal L_9$.\vspace{-6ex}}
\label{tab:data}
\end{table}
\end{center}

\bibliographystyle{alpha}
\bibliography{citation}{}
\end{document}